\documentclass[12pt]{article}
\usepackage[top=3cm, bottom=3cm, left=2cm, right=2cm]{geometry}
\usepackage[alphabetic,lite]{amsrefs}                                          % for bibliography
\usepackage{amssymb}                                                           % for math symbols
\usepackage{amsthm}                                                            % for math theorems
\usepackage{appendix}                                                          % for appendices
\usepackage{array}                                                             % for tabulation
\usepackage{bm}                                                                % for bold letters in math
\usepackage{colonequals}                                                       % for better ":="
\usepackage{color}                                                             % for colorful text
\usepackage{enumitem}                                                          % for enumerating
\usepackage{fullpage}                                                          % for setting margins
\usepackage{graphicx}                                                          % for inserting pictures
\usepackage{indentfirst}                                                       % for first line indent
\usepackage{mathrsfs}                                                          % for "mathscr"
\usepackage{mathtools}                                                         % for better "amsmath"
\usepackage{moreenum}                                                          % for enumerating by Greek letters
\usepackage{stmaryrd}                                                          % for (at least) maps from
\usepackage{verbatim}                                                          % for "verbatim"
\usepackage[all,cmtip]{xy}                                                     % for commutative diagrams

\definecolor{tianred}{rgb}{0.57, 0.36, 0.51}                                   % nice red %0.79, 0.17, 0.57
\definecolor{tianblue}{rgb}{0.0, 0.22, 0.66}                                   % nice blue
\definecolor{tianpink}{rgb}{0.88, 0.56, 0.59}                                  % nice pink
\definecolor{tiangreen}{rgb}{0.24, 0.82, 0.44}                                 % nice green

\usepackage[colorlinks,
            linkcolor=tianred,
            anchorcolor=tiangreen,
            citecolor=tianblue,
            urlcolor=tianpink]{hyperref}                                       % reference
\usepackage{cleveref}                                                          % reference

\theoremstyle{plain}      \newtheorem{thm}{Theorem}[section]                   % 定理-英
\theoremstyle{plain}      \newtheorem{thme}[thm]{Th{\'e}or{\`e}me}             % 定理-法
\theoremstyle{plain}      \newtheorem{lemme}[thm]{Lemme}                       % 引理-法
\theoremstyle{plain}                         % 推论-法
\theoremstyle{plain}      \newtheorem{prop}[thm]{Proposition}                  % 命题-英=法
\theoremstyle{plain}                   % 猜想-英=法
\theoremstyle{definition} \newtheorem{rmke}[thm]{Remarque}                     % 注释-法
\theoremstyle{definition} \newtheorem{dfe}[thm]{D\'efinition}                  % 定义-法
\theoremstyle{definition} \newtheorem{ege}[thm]{Exemple}                       % 例子-法
\theoremstyle{definition}                  % 假设-法
\theoremstyle{definition}                       % 习题-法
\theoremstyle{definition}                  % 符号-英=法
\theoremstyle{definition}          % 构造-英=法

\crefname{thme}{Th\'eor\`eme}{Th\'eor\`emes}
\crefname{lemme}{Lemme}{Lemmes}
\crefname{ege}{Exemple}{Exemples}
\crefname{rmke}{Remarque}{Remarques}
\crefname{core}{Corollaire}{Corollaires}
\crefname{dfe}{D\'efinition}{D\'efinitions}
\crefname{prop}{Proposition}{Propositions}

\newcommand{\benum}{\begin{enumerate}[label={{\upshape(\alph*)}}]}             % (a),(b),(c),etc
\newcommand{\benuma}{\begin{enumerate}[label={{\upshape(\arabic*)}}]}          % (1),(2),(3),etc
\newcommand{\benumr}{\begin{enumerate}[label={{\upshape(\roman*)}}]}           % (i),(ii),(iii),etc
\newcommand{\eenum}{\end{enumerate}}
\newcommand{\bitem}{\begin{itemize}}                                           % itemize
\newcommand{\eitem}{\end{itemize}}                                             % itemize

\newcommand{\ol}{\overline}                                                    % overline

\DeclareMathOperator{\Nrd}{Nrd}                                                % Reduced Norm

\DeclareMathOperator{\cd}{cd}                                                  % Cohomological dimension
\DeclareMathOperator{\scd}{scd}                                                % strict cohomological dimension
\DeclareMathOperator{\SK}{SK}                                                  % reduced Whitehead group
\DeclareMathOperator{\USK}{USK}                                                % reduced unitrary Whitehead group
\DeclareMathOperator{\GL}{\mathbf{GL}}                                         % general linear group
\newcommand{\ce}{\colonequals}                                                 % colon + equal
\newcommand{\itm}{\item}

\newcommand{\car}{\mathrm{car}}

\newcommand{\newpara}{\noindent\refstepcounter{thm}{\bf(\thethm)\;}}  %%%%%%%%%%%%%%%% new paragraph

\begin{document}
\title{\textbf{Trivialit\'e des groupes de Whitehead r\'eduits avec applications \`a l'approximation faible et l'approximation forte}}
\author{Yong HU et Yisheng TIAN}
\date{}

\maketitle %\tableofcontents

\begin{flushright}

\emph{D\'edi\'e \`a Professeur Jean-Louis Colliot-Th\'el\`ene}\\

\emph{\`a l'occasion de son 75-i\`eme anniversaire}

\end{flushright}

\renewcommand{\abstractname}{{\bf R\'esum\'e}}

\begin{abstract}
Nous prouvons certains r\'esultats de trivialit\'e pour des groupes de Whitehead r\'eduits et groupes de Whitehead unitaires r\'eduits pour des alg\`ebres \`a division sur un corps de valuation discr\`ete hens\'elien dont le corps r\'esiduel a pour dimension cohomologique virtuelle ou s\'eparable $\le 2$. Ces r\'esultats sont appliqu\'es pour d\'emontrer l'approximation forte pour des groupes simplement connexes absolument presque simples isotropes de type A. Comme cas particulier, un tel groupe d\'efini sur le corps des fonctions d'une courbe non r\'eelle  $C/k$ v\'erifie l'approximation forte si le corps de base $k$ est un corps de nombres, un corps $p$-adique, $\mathbb{C}(\!(t)\!)$ ou un corps de fonctions \`a deux variables sur $\mathbb{R}$.
\end{abstract}

\renewcommand{\abstractname}{{\bf Abstract}}

\begin{abstract}
We prove some triviality results for reduced Whitehead groups and reduced unitary Whitehead groups for division algebras over a henselian discrete valuation field whose residue field has virtual cohomological dimension or separable dimension $\le 2$. These results are applied to show strong approximation for isotropic absolutely almost simple simply connected groups of type A. In particular, such a group defined over the function field of a nonreal curve $C/k$ satisfies strong approximation if the base field $k$ is a number field, a $p$-adic field, $\mathbb{C}(\!(t)\!)$ or a two-variable function field over $\mathbb{R}$.
\end{abstract}


%19B28 $K_1$ of group rings and orders
%19Bxx		Whitehead groups and $K_1$
%19B99  	None of the above, but in this section
%11E57  	Classical groups
%20G35  	Linear algebraic groups over adeles and other rings and schemes
% 16K20  	Finite-dimensional division rings

Mots cl\'es:
 groupe de Whitehead r\'eduit,
groupe de Whitehead unitaire,
alg\`ebre \`a division sur un corps hens\'elien,
approximation faible,
approximation forte,
groupes simplement connexes

\section{Introduction}

Soit $A$ une alg\`ebre simple centrale (de dimension finie) sur un corps $K$.
D\'esignons par $A^*$ le groupe des unit\'es de $A$.
Le \emph{groupe de Whitehead} $K_1(A)$ est un quotient de $A^*/[A^*,\,A^*]$,
o\`u $[A^*,\,A^*]$ est le sous-groupe des commutateurs de $A^*$,
et la norme r\'eduite de $A$ induit un homomorphisme bien d\'efini $\Nrd_A: K_1(A)\to K^*$
(voir par exemple \cite{GSz17}*{\S\,2.10}).
Le \emph{groupe de Whitehead r\'eduit} de $A$ est d\'efini par
\[
\SK_1(A)\ce \mathrm{Ker}\big(\Nrd_A\,:\;K_1(A)\longrightarrow K^*\big)\,.
\]
Si $B$ est une alg\`ebre simple centrale sur $K$ qui est Brauer \'equivalente \`a $A$,
alors on a
$\SK_1(A)\cong \SK_1(B)$ par \cite{GSz17}*{Lemmas\;2.10.5 and 2.10.8}.
\`A l'exception de l'alg\`ebre matricielle  $A=\mathrm{M}_2(\mathbb{F}_2)$,
o\`u $\mathbb{F}_2$ est le corps fini de $2$ \'el\'ements,
on a un isomorphisme (voir \cite{Draxl}*{\S\;20, Thm.\;4 (iii)})
\[
\SK_1(A)\cong\frac{\mathrm{SL}_1(A)}{[A^*\,,\,A^*]}\quad\text{ o\`u }\; \mathrm{SL}_1(A)\ce \mathrm{Ker}\big(\Nrd_A\,:\;A^*\to K^*\big)\,.
\]
Il est bien connu que l'\'etude de $\SK_1(A)$ peut \^etre r\'eduite au cas o\`u $A$ est une alg\`ebre simple centrale \`a division de degr\'e une puissance d'un nombre premier (\cite{Draxl}*{\S\;23, Lemma\;6}).

On \'ecrit $\SK_1(A)=1$ si le groupe $\SK_1(A)$ est trivial.
Un th\'eor\`eme de Wang \cite{WangShianghaw50AJM} dit que
$\SK_1(A)=1$ si $K$ est un corps de nombres ou l'indice (de Schur) $\mathrm{ind}(A)$ est sans facteurs carr\'es.
Pour chaque nombre premier $p$, soit $\cd_p(K)$ la $p$-dimension cohomologique du corps $K$ (voir \cite{SerreCG94}*{\S I.3.1}).
La \emph{$p$-dimension cohomologique virtuelle} $\mathrm{vcd}_p(K)$ est d\'efinie par $\mathrm{vcd}_p(K)\ce \mathrm{cd}_p(K(\sqrt{-1}))$.
Yanchevskii a prouv\'e que si la caract\'eristique $\car(K)$ ne divise pas $\mathrm{ind}(A)$ et
$\mathrm{vcd}_p(K)\le 2$ pour chaque nombre premier $p\,|\,\mathrm{ind}(A)$,
alors $\SK_1(A)=1$ (\cite{Yanchevskii74MatSb} et \cite{Yanchevskii04Whitehead}).
Un analogue dans le cas $\car(K)\,|\,\mathrm{ind}(A)$ a \'et\'e \'etabli dans \cite{GB07}.

\medskip

Soit $(K,\,v)$ un corps valu\'e hens\'elien \`a corps r\'esiduel $\ol{K}$ et groupe de valeurs  $\Gamma_K=v(K^*)$.
Soient $p$ un nombre primier et $r_p\ce \dim_{\mathbb{F}_p}(\Gamma_K/p\Gamma_K)$
(ici, on consid\`ere $\Gamma_K$ comme un groupe additif).
Soit $D$ une alg\`ebre simple centrale \`a division sur $K$ de degr\'e une puissance de $p$.
Soman \cite{Soman19} a prouv\'e que $\SK_1(D)=1$ si
$\car(\ol{K})\neq p$, $1\le r_p\le 3$  et $\cd_p(\ol{K})\le 3-r_p$.

Le premier objectif du pr\'esent article est d'\'etendre le th\'eor\`eme de Soman dans deux cas.
L'un est le cas o\`u $p=\car(\ol{K})$ (\cref{1.1.5}),
et l'autre est le cas o\`u $v$ est une valuation discr\`ete
et $\mathrm{vcd}_p(\ol{K})\le 2$ (\cref{1.1.8}).
En particulier, nous pouvons conclure que $\SK_1(A)=1$ pour toutes les alg\`ebres simples centrales $A$ sur des corps comme $\mathbb{Q}(\!(t)\!)$,
$\mathbb{R}(x,\,y)(\!(t)\!)$, etc.
Comme produit secondaire, nous obtenons \'egalement la trivialit\'e de $\SK_1(A)$ pour les alg\`ebres de biquaternions sur certains corps ``semi-globaux'' (\cref{thm2.6}).

\medskip

Soit $C$ une courbe affine irr\'eductible normale sur un corps $k$.
De la trivialit\'e des groupes de Whitehead r\'eduits,
on obtient des r\'esultats d'approximation pour un groupe absolument presque simple et simplement connexe de type A int\'erieur,
c'est-\`a-dire, un groupe alg\'ebrique de la forme $G=\mathbf{SL}_n(D)$,
o\`u $n\ge 1$ et $D$ est une alg\`ebre simple centrale \`a division sur le corps des fonction $k(C)$.
Avec quelques hypoth\`eses sur $D$ ou $k$,
on prouve dans le \cref{thm4.3} que
$G$ satisfait l'approximation faible sur $C$ et que si $n\ge 2$,
$G$ satisfait m\^eme l'approximation forte.
En particulier, si $n\ge 2$ et $k$ est un corps de caract\'eristique $0$ avec $\mathrm{vcd}(k)\le 2$
(par exemple, $k$ est une extension finie de $\mathbb{Q}$, $\mathbb{Q}_p$, $\mathbb{R}(\!(t)\!)$, $\mathbb{R}(x)$ ou $\mathbb{R}(x,\,y)$),
alors $G$ satisfait l'approximation forte sur $C$.

Dans un autre travail avec J.\,Liu \cite{HuLiuTian23}, on utilise l'arithm\'etique des formes quadratiques enti\`eres pour donner un exemple
o\`u le groupe $\mathbf{SL}_1(Q)$ associ\'e \`a une certaine alg\`ebre de quaternions $Q$ sur $k(C)$ ne satisfait pas l'approximation forte.

\

La partie restante de cet article concerne des r\'esultats de trivialit\'e
pour le groupe de Whitehead unitaire r\'eduit et des applications aux propri\'et\'es d'approximation
pour les groupes semi-simples simplement connexes de type A ext\'erieur.
Les principaux r\'esultats dans cette direction sont les Th\'eor\`emes\;\ref{2.3.7} et \ref{thm4.7}.

\

Certains de nos r\'esultats sur les groupes de Whitehead r\'eduits (unitaires) et l'approximation faible ont \'et\'e obtenus par V.\,Suresh ind\'ependamment.

\section{Les groupes de Whitehead r\'eduits}\label{sec2}

Dans cette section,
soit $(K,\,v)$ un corps valu\'e hens\'elien \`a corps r\'esiduel $\overline{K}$ et groupe de valeurs $\Gamma_K$.
Soient $p$ un nombre premier et $r_p=\dim_{\mathbb{F}_p}(\Gamma_K/p\Gamma_K)$.

\medskip

\newpara Nous rappelons quelques notations de \cite{Soman19}*{\S\;2}.
Soit $D$ une alg\`ebre simple centrale \`a division sur $K$.
La valuation $v$ s'\'etend uniquement \`a une valuation sur $D$.
On \'ecrit  $\ol{D}$ et $\Gamma_D$ pour l'alg\`ebre  \`a division r\'esiduelle et le groupe de valeurs de $D$ respectivement.
L'application canonique $\theta_D:\Gamma_D\to \mathrm{Aut}(Z(\ol{D})/\ol{K})$ est d\'efinie dans
\cite{TignolWadsworth15}*{\S\;1.1.1, p.~3},
o\`u $Z(\ol{D})$ est le centre de $\ol{D}$.
On sait que $\theta_D$ est surjective et que l'extension $Z(\ol{D})/\ol{K}$ est quasi-galoisienne (\cite{TignolWadsworth15}*{Prop.\;1.5}).

Suivant \cite{JW90}*{\S\;6},
on dit que $D$ est \emph{mod\'er\'ee} (``\emph{tame}'' en anglais) sur $K$ si $D$ est sans d\'efaut (c'est-\`a-dire $[D:K]=[\ol{D}:\ol{K}][\Gamma_D:\Gamma_K]$),
l'extension $Z(\ol{D})/\ol{K}$ est s\'eparable (donc abelien) et $\car(\ol{K})\nmid [\mathrm{Ker}(\theta_D):\Gamma_K]$.

Si $\car(\ol{K})\nmid\deg(D)$, alors $D$ est  mod\'er\'ee par \cite{TignolWadsworth15}*{Prop.\;4.9}.

Lorsque $D$ est mod\'er\'ee, le nombre
\[
\zeta\ce \frac{\deg(D)}{\deg(\ol{D})[Z(\ol{D}):\ol{K}]}
\]
est un entier. En effet,
\[
\zeta^2
=\frac{[D:K]}{[\ol{D}:Z(\ol{D})][Z(\ol{D}):\ol{K}]^2}
=\frac{[\Gamma_D:\Gamma_K]}{[Z(\ol{D}):\ol{K}]}
=[\mathrm{Ker}(\theta_D):\Gamma_K]\;\in\,\mathbb{N}\,.
\]Notons $\mu_{\zeta}(\ol{K})$  le groupe des  racines $\zeta$-i\`emes de l'unit\'e dans $\ol{K}$.

%Notons $\tilde{N}:\ol{D}\to\ol{K}$ la compos\'ee de la norme $N_{Z(\ol{D})/\ol{K}}$ pour l'extension $Z(\ol{D})/\ol{K}$ et la norme r\'eduite $\Nrd_{\ol{D}}:\ol{D}\to Z(\ol{D})$.
%Posons $C:=\tilde{N}(\ol{D}^*)\cap \mu_{\zeta}(\ol{K})$,
%Soit $v_D$ l'extension unique \`a $D$ de la valuation $v$ de $K$. Posons
%\[
%\begin{split}
%  U&:=\{x\in D^*\,|\,v_D(x)=0\}\,,\\
%  \mathrm{SL}^{v_D}(D)&:=\{x\in\mathrm{SL}_1(D)\,|\,\tilde{N}(\bar x)=1\}\,,\\
%  C&:=\tilde{N}(\ol{D}^*)\cap \mu_{\zeta}(\ol{K})\,,
%\end{split}
%\]
%o\`u $\mu_{\zeta}(\ol{K})$ d\'esigne le groupe des $\zeta$-i\`emes racines de l'unit\'e dans $\ol{K}$.

\medskip

On utilisera le lemme suivant, d\^u \`a Ershov; voir \cite{Ershov82MathSb}*{p.~68} ou \cite{Soman19}*{Thm.\;2.1}.

\begin{lemme}\label{2p2new}
Avec les notations ci-dessus,  soit $\ell=Z(\ol{D})$. Lorsque  $D$ est mod\'er\'ee, pour un certain sous-groupe $C\le\mu_{\zeta}(\ol{K})$ et deux groupes  convenables $G_1,G_2$, on a trois suites exactes comme suit:
\[
  \begin{split}
    \SK_1(\ol{D})\longrightarrow G_1 \longrightarrow  &\,\hat{H}^{-1}(\ell/\ol{K}\,,\,\Nrd_{\ol{D}}(\ol{D}^*)) \longrightarrow 1\,,\\
    1 \longrightarrow G_1 \longrightarrow &\,G_2 \longrightarrow C \longrightarrow 1\,,\\
     &\,G_2  \longrightarrow \SK_1(D) \longrightarrow 1\,.
  \end{split}
\](On \'ecrit simplement $\ell/\ol{K}$ au lieu du groupe de Galois $\mathrm{Gal}(\ell/\ol{K})$ dans la notation du groupe de cohomologie, par convention standard en cohomologie galoisienne.)
\end{lemme}

%\begin{lemme}\label{2p2new}
%Avec les notations ci-dessus,  soit $\ell=Z(\ol{D})$. Lorsque  $D$ est mod\'er\'ee, on a trois suites exactes comme suit:
%\[
%  \begin{split}
%    1 \longrightarrow \mathrm{SL}^{v_D}(D)/[U,D^*] \longrightarrow &\,\mathrm{SL}_1(D)/[U,D^*] \longrightarrow C \longrightarrow 1\,,\\
%     \SK_1(\ol{D})\longrightarrow \mathrm{SL}^{v_D}(D)/[U,D^*] \longrightarrow  &\,\hat{H}^{-1}(\ell/\ol{K}\,,\,\Nrd_{\ol{D}}(\ol{D}^*)) \longrightarrow 1\,,\\
%    1\longrightarrow [D^*,D^*]/[U,D^*]\longrightarrow &\,\mathrm{SL}_1(D)/[U,\,D^*]  \longrightarrow \SK_1(D) \longrightarrow 1\,.
%  \end{split}
%\](On \'ecrit simplement $\ell/\ol{K}$ au lieu du groupe de Galois $\mathrm{Gal}(\ell/\ol{K})$ dans la notation du groupe de cohomologie, par convention standard en cohomologie galoisienne.)
%\end{lemme}

Le r\'esultat suivant g\'en\'eralise \cite{Soman19}*{Thm.\;1.1} dans le cas $r_p=1$.

\begin{thme}\label{1.1.8}
Soit $D$ une alg\`ebre simple centrale \`a division sur $K$ de degr\'e une puissance d'un nombre premier $p$.
Supposons $\car(\ol{K})\neq p$, $r_p=1$ et
$\mathrm{vcd}_p(\ol{K})=\cd_p(\ol{K}(\sqrt{-1}))\le 2$.
Alors $\SK_1(D)=1$.
\end{thme}
\begin{proof}[D\'emonstration]
Rappelons qu'un corps $k$ est appel\'e \emph{(formellement) r\'eel} ou \emph{ordonnable} si $-1$ n'est pas une somme d'un nombre fini de carr\'es dans $k$.
On peut supposer que $p=2$ et $\ol{K}$ est r\'eel,
car sinon $\cd_p(\ol{K})=\mathrm{vcd}_p(\ol{K})$ d'apr\`es \cite{SerreCG94}*{\S\;II.4, Prop.\;10'}
et le r\'esultat est d\'ej\`a connu dans le th\'eor\`eme de Soman.

Soit $\ell=Z(\ol{D})$. Comme $D$ est mod\'er\'ee, l'extension $\ell/\ol{K}$ est galoisienne.  On a $\mathrm{vcd}_2(\ell)\le \mathrm{vcd}_2(\ol{K})\le 2$. Par \cite{Yanchevskii04Whitehead}*{Thm.\;3.1}, on voit que $\SK_1(\ol{D})=1$.

L'hypoth\`ese $r_p=1$ entra\^ine que le groupe $\Gamma_D/\Gamma_K$ est cyclique, d'apr\`es \cite{Soman19}*{Lemma\;4.2}. Ainsi, $\mathrm{Ker}(\theta_D)/\Gamma_K$ est cyclique aussi, et la surjectivit\'e de $\theta_D$ implique que l'extension $\ell/\ol{K}$ est cyclique. D'autre part, \cite{TignolWadsworth15}*{Prop.\;8.17 (iv)} nous dit que le groupe $\mathrm{Ker}(\theta_D)/\Gamma_K$ admet un accouplement altern\'e non-d\'eg\'en\'er\'e. Ce groupe, \'etant cyclique, doit alors \^etre trivial. Cela montre que $\zeta=1$. Le groupe $C\le \mu_{\zeta}(\ol{K})$ est donc trivial.

D'apr\`es le Lemme \ref{2p2new}, il suffit de montrer que $\hat{H}^{-1}(\ell/\ol{K}\,,\,\Nrd_{\ol{D}}(\ol{D}^*))=1$.

 Si $\ell$ n'est pas r\'eel, alors $\cd_2(\ell)=\mathrm{vcd}_2(\ell)\le 2$. Cela donne $\Nrd_{\ol{D}}(\ol{D}^*)=\ell^*$, par un th\'eor\`eme bien connu de Merkurjev--Suslin \cite{MerkurjevSuslin82English}. Ainsi, on obtient
 \[
 \hat{H}^{-1}(\ell/\ol{K},\,\Nrd_{\ol{D}}(\ol{D}^*))=\hat{H}^{-1}(\ell/\ol{K},\,\ell^*)=H^1(\ell/\ol{K},\,\ell^*)=1
  \]par le th\'eor\`eme $90$ de Hilbert. (Ici on utilise the fait que $\ell/\ol{K}$ est cyclique.)

Supposons maintenant $\ell$  r\'eel. Soit $S_0$ l'ensemble des ordres $i$ de $\ell$ tels que $\ol{D}$ n'est pas d\'eploy\'ee sur la cl\^oture r\'eelle de $\ell$ par rapport \`a $i$.
D'apr\`es \cite{BFP98}*{Thm.\;2.1}, on a
\[
\Nrd_{\ol{D}}(\ol{D}^*)=\ell^+\ce \{a\in \ell\,|\,i(a)>0 \text{ pour tout }\,i\in S_0\}\,.
\]
(Ici, on identifie un ordre avec un plongement dans un corps r\'eellement clos.)
On d\'efinit $S:=\{i\sigma\,|\,i\in S_0,\,\sigma\in \mathrm{Gal}(\ell/\ol{K})\}$.
Comme expliqu\'e dans \cite{Ershov82MathSb}*{p.~66}, $\Nrd_{\ol{D}}(\ol{D}^*)$ est stable sous l'action galoisienne de $\mathrm{Gal}(\ell/\ol{K})$.
Cela implique
\[
\Nrd_{\ol{D}}(\ol{D}^*)=\ell^+=\ell_S^+\ce \{a\in \ell\,|\,i(a)>0 \text{ pour tout }\,i\in S\}\,.
\]
Il reste \`a prouver $\hat{H}^{-1}(\ell/\ol{K}\,,\,\ell_S^+)=1$, que l'on fait dans le \cref{1.1.7} (3) ci-dessous.
\end{proof}

\begin{lemme}\label{1.1.7}
Soit $L/k$ une extension finie galoisienne de corps r\'eels avec $G=\mathrm{Gal}(L/k)$.
Soit $S$ un ensemble d'ordres de $L$ qui est stable sous l'action naturelle de $G$,
c'est-\`a-dire, $i\sigma\in S$ pour tout $i\in S$ et tout $\sigma\in G$.
Soient $H\lhd G$ un sous-groupe distingu\'e,
$E:=L^H$ le sous-corps fix\'e et $T$ l'ensemble des ordres de $E$ obtenus par restrictions des ordres dans $S$ \`a $E$.

\benuma
\item\label{1.1.7.1}
L'ensemble $S$ est pr\'ecis\'ement l'ensemble des ordres de $L$ qui \'etendent les ordres dans $T$,
et $T$ est stable sous l'action naturelle de $\mathrm{Gal}(E/k)=G/H$.

\item\label{1.1.7.2}
On d\'efinit $L_S^+\ce \{a\in L^*\,|\,i(a)>0 \text{ pour tout }\; i\in S\}$.

Alors $L_S^+$ est stable sous l'action naturelle de $G$ sur $L^*$, et
\[
(L_S^+)^H=E_{T}^+\ce \{x\in E^*\,|\,i(x)>0 \text{ pour tout }\; i\in T\}\,.
\]

\item\label{1.1.7.3}
Supposons que $[L:k]$ est une puissance de $2$.
On a $H^1(L/k,\,L_S^+)=1$.
\end{enumerate}
\end{lemme}
\begin{proof}[D\'emonstration]
\hfill
\benuma
\itm
La premi\`ere affirmation est un fait standard (voir par exemple \cite{Lam05}*{Chap.\;VIII, Appendix A,\,Cor.\;2.20}).
La deuxi\`eme est une cons\'equence de la d\'efinition.

\itm
Ceci est une cons\'equence de la stabilit\'e de $S$ sous l'action de $G$.

\itm
On peut choisir comme $H$ un sous-groupe d'indice 2 dans $G$.
Par \ref{1.1.7.2},
la suite exacte de restriction-inflation pour le $G$-module $L_S^+$ est
\[
1\longrightarrow H^1(E/k,\,E_T^+)\longrightarrow H^1(L/k,\,L_S^+)\longrightarrow H^1(L/E,\,L_S^+)\,.
\]
Par r\'ecurrence sur $[L:k]$, on se r\'eduit au cas $[L:k]=2$.
Selon l'isomorphisme $H^1(L/k,\,L_S^+)\cong \hat{H}^{-1}(L/k,\,L_S^+)$,
il suffit de montrer que pour tout $a\in L_S^+$ tel que $N_{L/k}(a)=1$,
il existe $b\in L_S^+$ tel que $a=b/\sigma(b)$,
o\`u $\sigma$ est l'\'el\'ement non trivial du groupe $G=\mathrm{Gal}(L/k)$.
Pour ce faire, observons d'abord qu'il existe un \'el\'ement $c\in L^*$ tel que $a=c/\sigma(c)$.
Comme $a\in L_S^+$, pour chaque $i\in S$ les \'el\'ements $c$ et $\sigma(c)$ ont le m\^eme signe par rapport \`a l'ordre $i$.
Donc $t\ce c+\sigma(c)\in k$ est du m\^eme signe que $c$ par rapport \`a chaque $i\in S$.
En prenant $b=tc$, on obtient $a=b/\sigma(b)$ et $b\in L_S^+$.
Ceci compl\`ete la d\'emonstration.
\qedhere
\end{enumerate}
\end{proof}

\begin{rmke}\label{remark2.4}
Sans aucune restriction sur $r_p$ et $D$,
on peut montrer que $\SK_1(D)=1$ si $\ol{K}=\mathbb{R}$ (\cite{TignolWadsworth15}*{Exercise\;11.1})
ou si $K=k(\!(t_1)\!)\cdots (\!(t_n)\!)$, o\`u $k$ est un corps de fonctions \`a une variable sur $\mathbb{R}$ (\cite{Lipnitskii82}).

Si $D$ est une alg\`ebre simple centrale \`a division non ramifi\'ee
(ou ``inertial'' en terminologie de \cite{JW90}*{p.~138} et \cite{TignolWadsworth15}*{Chap.\;8}) sur $K$,
on a $\SK_1(D)\cong\SK_1(\ol{D})$ par \cite{TignolWadsworth15}*{Thm.\;11.23 (i)}.
Dans ce cas, on a $\SK_1(D)=1$ si $\SK_1(\ol{D})=1$
(par exemple si $\ol{K}=\mathbb{Q}(\!(t)\!)$, par le \cref{1.1.8}).
\end{rmke}

\begin{rmke}\label{remark2.5new}
Soit $D$ une alg\`ebre simple centrale \`a division  sur un corps $F$ de caract\'eristique $\neq 2$.
On sait que $\SK_1(D)=1$ dans chacun des cas suivants :

\benuma
\item\label{remark2.5new.1}
$\deg(D)=4$ et $\mathrm{cd}_2(F)\le 3$.

\item\label{remark2.5new.2}
$D$ est une alg\`ebre de biquaternions et $F$ est le corps des fonctions d'une courbe sur un corps de nombres.
\end{enumerate}

Le cas \ref{remark2.5new.1} r\'esulte de \cite{Merkurjev99crelle}*{Thm.\;6.6} ou \cite{Suslin06SK1}*{p.~134, Thm.\;3}.
Le cas \ref{remark2.5new.2} a \'et\'e prouv\'e dans \cite{Colliot96Ktheory}*{Thm.\;3.1}.
\end{rmke}

Nous observons le r\'esultat suivant du m\^eme type que la \cref{remark2.5new} \ref{remark2.5new.2}.

\begin{thme}\label{thm2.6}
Soit $k$ un corps r\'eel tel que $\mathrm{vcd}_2(k)\le 1$
(par exemple $k=\mathbb{R}(t)$ ou $\mathbb{R}(\!(t)\!)$).
Soit $F$ une extension finie de $k(\!(x,\,y)\!)$ ou de $k(\!(x)\!)(y)$.

Alors $\SK_1(A)=1$ pour chaque alg\`ebre de biquaternions  $A$ sur $F$.
\end{thme}
\begin{proof}[D\'emonstration]
Soit $\Omega$ l'ensemble $\mathcal{X}^{(1)}$ d\'efini comme dans \cite{Hu2017IMRN}*{Thm.\;1.2}.
Pour chaque valuation discr\`ete $v\in\Omega$,
son corps r\'esiduel $\kappa(v)$ est isomorphe \`a une extension finie de $k(\!(x)\!)$ ou de $k(x)$.
On a donc $\mathrm{vcd}_2(\kappa(v))\le 2$.
Par le \cref{1.1.8}, sur le compl\'et\'e $F_v$ l'alg\`ebre $A_v\ce A\otimes_FF_v$ satisfait $\SK_1(A_v)=1$ pour toute $v\in\Omega$.
Maintenant le r\'esultat est une cons\'equence de \cite{Hu2017IMRN}*{Thm.\;1.2} et du diagramme commutatif
\[
\xymatrix{
\SK_1(A) \ar[d]\ar[rr] && \prod_{v\in\Omega}\SK_1(A_v) \ar[d] \\
H^4(F,\,\mathbb{Z}/2) \ar[rr] && \prod_{v\in\Omega} H^4(F_v,\,\mathbb{Z}/2)
}
\]
o\`u les fl\`eches verticales sont les invariants de Rost (qui sont fonctoriels et injectifs, voir \cite{Merkurjev95ProcSympos58}*{Thm.\;4}).
\end{proof}

\newpara\label{para2.4new}
Pour tout corps $k$, la \emph{$p$-dimension s\'eparable} $\scd_p(k)$ est d\'efinie dans \cite{Gille00}*{p.~62}.
Si $\car(k)\neq p$, alors $\scd_p(k)=\cd_p(k)$.
En utilisant la d\'efinition et la th\'eorie de Galois, on peut montrer que $\scd_p(k)=0$ si et seulement si
pour toute extension s\'eparable finie $\ell/k$,
le corps $\ell$ n'a pas d'extension s\'eparable de degr\'e une puissance de $p$.
D'apr\`es \cite{Gille19book}*{Prop.\;4.6.1 et Thm.\;4.7.1}, on a :

\begin{itemize}
\item
$\scd_p(k)\le 1$ si et seulement si pour toute extension s\'eparable finie $\ell/k$,
l'application de norme $N_{E/\ell}: E^*\to \ell^*$ est surjective pour toute extension s\'eparable finie $E/\ell$.

\item
$\scd_p(k)\le 2$ si et seulement si pour toute extension s\'eparable finie $\ell/k$,
l'application de norme r\'eduite $\Nrd_{B/\ell}: B^*\to \ell^*$ est surjective
pour chaque alg\`ebre simple centrale $B$ sur $\ell$ d'indice une puissance de $p$.
\end{itemize}
%
%Rappelons que $D$ est dite \emph{totalement ramifi\'ee} si
%
%\medskip
%
%Le lemme suivant est un analogue du Lemme 4.3 dans \cite{Soman19}.
%
%
%\begin{lemme}\label{2p9new}
%Supposon $\car(\ol{K})=p$ et $2\le r_p\le 3$. Soit $T$ une alg\`ebre simple centrale  \`a division mod\'er\'ee et totalement ramifi\'ee de degr\'e une puissance de $p$.
%
%Alors $T$ est une alg\`ebre de symbole et $\SK_1(T)=1$.
%\end{lemme}
%\begin{proof}[D\'emonstration]
%Soit $\deg(T)=p^t$.
%\end{proof}

Maintenant, on \'etend \cite{Soman19}*{Thm.\;1.1} au cas $p=\car(\ol{K})$.

\begin{thme}\label{1.1.5}
Soit $D$ une alg\`ebre simple centrale  \`a division {\bf mod\'er\'ee} de degr\'e une puissance de $p$.
Supposons $\car(\ol{K})=p$, $1\le r_p\le 3$ et $\mathrm{scd}_p(\ol{K})\le 3-r_p$.
Alors $\SK_1(D)=1$.
\end{thme}

\begin{proof}[D\'emonstration]
Les id\'ees cruciales sont similaires \`a  celles dans la preuve du th\'eor\`eme de Soman.
 Remarquons d'abord que comme $\zeta$ est une puissance de $p$ et $\car(\ol{K})=p$, le groupe  $\mu_{\zeta}(\ol{K})$ et son sous-groupe $C$ sont nuls.
(Gr\^ace \`a ce fait, notre d\'emonstration devient plus courte que celle de Soman.)

 Par la vertu des trois suites exactes du Lemme \ref{2p2new}, il reste \`a d\'emontrer que les groupes $\SK_1(\ol{D})$ et $\hat{H}^{-1}(\ell/\ol{K},\,\Nrd_{\ol{D}}(\ol{D}^*))$ s'annulent, o\`u $\ell=Z(\ol{D})$.

\bitem
\itm
Le cas $r_p=1$. Par l'hypoth\`ese, on a $\scd_p(\ol{K})\le 2$. Ainsi $\scd_p(\ell)\le\scd_p(\ol{K})\le 2$. On obtient $\Nrd_{\ol{D}}(\ol{D}^*)=\ell^*$ par \eqref{para2.4new},
et $\SK_1(\ol{D})=1$ \`a l'aide de \cite{GB07}*{Cor.\;5.5}. Comme on a vu dans la preuve du Th\'eor\`eme \ref{1.1.8}, $\ell/\ol{K}$ est une extension cyclique, et donc la trivialit\'e de $\hat{H}^{-1}(\ell/\ol{K},\,\ell^*)$ r\'esulte du th\'eor\`eme 90 de Hilbert.

\itm
Le cas $2\le r_p\le 3$. Dans ce cas, on a $\scd_p(\ell)\le \scd_p(\ol{K})\le 1$. Cela signifie que le groupe de Brauer de $\ell$ n'a pas de $p$-torsion.  Il s'ensuit que   $\ol{D}=\ell$ et donc $\SK_1(\ol{D})=1$. De plus, en tant que $\mathrm{Gal}(\ell/\ol{K})$-module, $\ell^*$ est cohomologiquement trivial par \cite{SerreCL}*{IX.5, Thm.\;8}. Donc,
 $\hat{H}^{-1}(\ell/\ol{K},\,\Nrd_{\ol{D}}(\ol{D}^*))=\hat{H}^{-1}(\ell/\ol{K},\,\ell^*)=1$.
\qedhere
\eitem
\end{proof}

\section{Les groupes de Whitehead unitaires r\'eduits}

Dans cette section, soit $K/F$ une extension quadratique s\'eparable de corps.

\medskip

\newpara\label{1.2.1}
Soit $A$ une alg\`ebre simple centrale sur $K$ qui admet une $K/F$-involution unitaire $\sigma$
(c'est-\`a-dire une involution de  seconde esp\`ece telle que $F$ est le sous-corps des \'el\'ements fix\'es par $\sigma$ dans $K$).
On d\'efinit les groupes
\[
\begin{split}
\Sigma_{\sigma}(A):&=\text{ le groupe engendr\'e par tous les \'el\'ements } a\in A^* \text{ tels que } \sigma(a)=a \,,\\
\Sigma'_{\sigma}(A):&=\left\{a\in A^*\,|\,\Nrd_{A/K}(a)=\sigma\big(\Nrd_{A/K}(a)\big)\right\}\,.
\end{split}
\]
Si $\tau$ est une autre $K/F$-involution sur $A$,
il existe $u\in A^*$ tel que $\sigma(u)=u=\tau(u)$ et $\tau=\mathrm{Int}(u)\circ\sigma$,
o\`u $\mathrm{Int}(u)$ est l'automorphisme int\'erieur $x\mapsto uxu^{-1}$ (\cite{KMRT98}*{(2.18)}).
Il s'ensuit que
$\Sigma_{\sigma}(A)=\Sigma_{\tau}(A)$ et
$\Sigma'_{\sigma}(A)=\Sigma'_{\tau}(A)$
(\cite{Yanchevskii74MatSb}*{Lemma\;1}).

Le \emph{groupe de Whitehead unitaire r\'eduit} de $A/F$ est d\'efini par
\[
\USK_1(A/F)\ce \USK_1(\sigma,A)\ce \Sigma'_{\sigma}(A)/\Sigma_{\sigma}(A)\,,
\]
o\`u $\sigma$ est une $K/F$-involution unitaire quelconque sur $A$.

\begin{rmke}\label{remark on USK1 and SK1}
Les faits suivants r\'ev\`elent certaines analogies entre les groupes $\USK_1(A/F)$ et $\SK_1(A)$ :

\benuma
\item
Si $A=A_1\otimes_KA_2$,
o\`u $A_1$ et $A_2$ sont des alg\`ebres simples centrales sur $K$ dont les indices sont premi\`eres entre eux, alors
$\USK_1(A/F)\cong\USK_1(A_1/F)\times\USK_1(A_2/F)$ (\cite{Yanchevskii78}*{Prop.\;2.7}).

\item
Si $D$ est l'alg\`ebre simple centrale \`a division sur $K$  Brauer \'equivalent \`a $A$,
alors on a $\USK_1(A/F)\cong \USK_1(D/F)$ par \cite{Yanchevskii74MatSb}*{Lemma\;3}.

\item
On a $\USK_1(A/F)=1$ dans chacun des cas suivants :

\benumr
\item
$\mathrm{ind}(A)$ est sans facteurs carr\'es (\cite{Yanchevskii78}*{Lemma\;2.8}).

\item
$K$ est un corps global (\cite{PlatonovYanchevskii73}).

\item
Pour chaque diviseur premier $p$ de $\mathrm{ind}(A)$,
$\scd_p(K)\le 2$ (\cite{Yanchevskii74MatSb}*{Thm.\;1} et \cite{GB07}*{Thm.\;1.1 and Cor.\;5.5}).

\item $\car(k)=0$ et $\mathrm{vcd}(k)\le 2$ (\cite{Yanchevskii04Whitehead}*{Thm.\;3.1}).
\qedhere
\end{enumerate}
\end{enumerate}
\end{rmke}
En particulier,
on peut se limiter \`a consid\'erer les alg\`ebres \`a division de degr\'e une puissance d'un nombre premier lors de l'\'etude des groupes de Whitehead unitaires r\'eduits.

\begin{thme}[{\cite{Yanchevskii78}*{Cor.\;4.10, 4.13 and 4.14}}]\label{2.3.5}
Soit $K$ un corps de valuation discr\`ete hens\'elien et
supposons que l'extension $\ol{K}/\ol{F}$ est s\'eparable.
Soit $D$ une alg\`ebre simple centrale \`a division {\bf mod\'er\'ee}  sur $K$ admettant une  $K/F$-involution.

\benuma
\item Si $D/K$ et $K/F$ sont   non ramifi\'ees,  on a $\USK_1(D/F)\cong\USK_1(\ol{D}/\ol{F})$.

\item Si  $K/F$ est totalement ramifi\'ee, ou si $D/K$ est ramifi\'ee et  $\ol{D}$ est commutative, alors $\USK_1(D/F)=1$.
\end{enumerate}
\end{thme}

\newpara\label{3p4new} Supposons que  $K/F$ est une extension quadratique non ramifi\'ee de corps de valuation discr\`ete hens\'eliens avec l'extension de corps r\'esiduels $\ol{K}/\ol{F}$ (s\'eparable). Soit $D$ une $K$-alg\`ebre simple centrale \`a division admettant une  $K/F$-involution. Supposons que $D/K$ est mod\'er\'ee  et ramifi\'ee.

Comme l'alg\`ebre r\'esiduelle $\ol{D}$ est s\'eparable sur $\ol{K}$, il existe une sous-alg\`ebre $E\subseteq D$ non ramifi\'ee sur $K$ telle que $\ol{E}=\ol{D}$ (\cite{JW90}*{Thm.\;2.9}). On choisit une telle sous-alg\`ebre $E\subseteq D$. D'apr\`es \cite{Yanchevskii78}*{Prop.\;3.17}, on peut choisir  une $K/F$-involution convenable $\tau$ sur $D$ et une certaine uniformisante $\Pi$ de $D$ telles que $\tau(E)=E$, $\tau(\Pi)=\Pi$ et $\mathrm{Int}(\Pi)(E)=E$, o\`u $\mathrm{Int}(\Pi): D\to D$ d\'esigne l'automorphisme int\'erieur $x\mapsto \Pi x\Pi^{-1}$. Posons $\tau_2=\mathrm{Int}(\Pi)\circ\tau$.

La restriction $\tau|_K$ est l'automorphisme non trivial de $K/F$. Comme $K/F$ est non ramifi\'ee, la r\'eduction $\ol{\tau}$ de $\tau$ est non triviale sur $\ol{K}$. Ainsi, $\ol{\tau}$ est une involution unitaire sur $\ol{D}$. Comme $\tau_2|_K=\tau|_K$, la r\'eduction $\ol{\tau}_2$ est aussi une involution unitaire sur $\ol{D}$.

Comme dans \cite{Yanchevskii78}*{4.11}, on dit qu'un \'el\'ement
\[
\tilde{a}\;\in\; \frac{Z(\ol{D})^*}{\ol{K}^*\Nrd_{\ol{D}/Z(\ol{D})}\big(\Sigma_{\ol{\tau}}(\ol{D})\big)}
\]est une {\bf conorme unitaire projective} s'il existe un r\'epresentant $a\in Z(\ol{D})^*$ de $\tilde{a}$ et un \'el\'ement $b\in\Nrd_{\ol{D}}(\ol{D}^*)$ tels que
$\sigma(a)/a=b/\ol{\tau}_2(b)$, o\`u $\sigma$ d\'esigne le g\'en\'erateur du groupe cyclique $\mathrm{Gal}(Z(\ol{D})/\ol{K})$ obtenue par restriction de $\ol{\mathrm{Int}(\Pi)}$.

Les conormes unitaires projectives forment un groupe, que l'on note $\mathrm{PU}(\tau,\,D)$.

\begin{prop}[{\cite{Yanchevskii78}*{Thm.\;4.12}}]\label{3p5new}
  Avec les notations et les hypoth\`eses du paragraphe \eqref{3p4new}, %pour un certain quotient $Q$ du groupe $\USK_1(\ol{\tau}_2,\,\ol{D})$ on a une suite exacte
  il existe une suite exacte
 % \[
%   1\longrightarrow Q\longrightarrow \USK_1(\tau,\,D)\longrightarrow \mathrm{PU}(\tau,\,D)\longrightarrow 1\,.
%  \]
  \[
   \USK_1(\ol{\tau}_2,\,\ol{D})\longrightarrow \USK_1(\tau,\,D)\longrightarrow \mathrm{PU}(\tau,\,D)\longrightarrow 1\,.
  \]
\end{prop}

Le th\'eor\`eme suivant est essentiellement une r\'ep\'etition de \cite{Yanchevskii78}*{Cor.\;4.15}.

\begin{thme}\label{3p6new}
 Soit $K$ un corps de valuation discr\`ete hens\'elien et supposons que l'extension $\ol{K}/\ol{F}$ est s\'eparable.
Soit $D$ une alg\`ebre simple centrale \`a division {\bf mod\'er\'ee}  sur $K$ admettant une  $K/F$-involution.

Si $\scd_p(\ol{K})\le 2$ pour chaque diviseur premier $p$ de $\deg(D)$, alors $\USK_1(D/F)=1$.
\end{thme}
\begin{proof}[D\'emonstration]
La preuve suit les m\^emes lignes que celles dans la preuve de Yanchevskii.  En effet, si $D/K$ et $K/F$ sont non ramifi\'ees, l'hypoth\`ese sur la dimension s\'eparable implique que $\USK_1(\ol{D}/\ol{F})=1$, par la Remarque \ref{remark on USK1 and SK1} (3)(iii). Vu le Th\'eor\`eme \ref{2.3.5}, on peut supposer que $\ol{D}$ est non commutative et que $K/F$ est non ramifi\'ee.

 Prenons les notations du paragraphe \eqref{3p4new}. On a $\USK_1(\ol{\tau}_2,\,\ol{D})=1$, par la Remarque \ref{remark on USK1 and SK1} (3)(iii). Selon la Proposition \ref{3p5new}, il suffit de montrer que $\mathrm{PU}(\tau,\,D)=1$.

 Soient $\ell=Z(\ol{D})$ et $\ell_0\subseteq \ell$ le sous-corps fix\'e par $\ol{\tau}$.  La Remarque \ref{remark on USK1 and SK1} (3)(iii) nous garanti aussi que
  $\Sigma_{\ol{\tau}}(\ol{D})=\Sigma'_{\ol{\tau}}(\ol{D})$. En outre, $\Nrd_{\ol{D}}(\Sigma'_{\ol{\tau}}(\ol{D}))=\ell_0^*$ puisque   $\Nrd_{\ol{D}}(\ol{D}^*)=\ell^*$ par \eqref{para2.4new}. D'o\`u
  \[
  \frac{Z(\ol{D})^*}{\ol{K}^*\Nrd_{\ol{D}/Z(\ol{D})}\big(\Sigma_{\ol{\tau}}(\ol{D})\big)}=\frac{\ell^*}{\ol{K}^*\ell_0^*}\;.
  \]Soit $a\in\ell^*$ un \'el\'ement tel que  $\sigma(a)/a=b/\ol{\tau}_2(b)$ pour un certain $b\in\ell^*=\Nrd_{\ol{D}}(\ol{D}^*)$. Notre but est de d\'emontrer que $a\in \ol{K}^*\ell_0^*$.

En fait, on a
  \[
  \ol{\tau}_2\sigma(a)\ol{\tau}_2(a)^{-1}\sigma(a)a^{-1}=\ol{\tau}_2(\sigma(a)a^{-1})\cdot (\sigma(a)a^{-1})=\ol{\tau}_2(b\ol{\tau}_2(b)^{-1})\cdot b\ol{\tau}_2(b)^{-1}=1\,.
  \]Par d\'efinition, $\ol{\tau}_2\sigma=\ol{\tau}$ et $\ol{\tau}_2=\sigma\ol{\tau}$. On obtient ainsi
  \[
  \sigma\big(a\ol{\tau}(a)^{-1}\big)\cdot\big(a\ol{\tau}(a)^{-1}\big)^{-1}=\sigma(a)(\sigma\ol{\tau}(a))^{-1}\ol{\tau}(a)a^{-1}
  =\ol{\tau}(a)(\sigma\ol{\tau}(a))^{-1}\sigma(a)a^{-1}=1\,.
  \]Comme $\sigma$ est un g\'en\'erateur de $\mathrm{Gal}(\ell/\ol{K})$, il s'ensuit que $a\ol{\tau}(a)^{-1}\in \ol{K}^*$. D'autre part, comme $K/F$ est non ramifi\'ee, $\ol{\tau}|_{\ol{K}}$ est l'unique \'el\'ement non trivial du groupe $\mathrm{Gal}(\ol{K}/\ol{F})$. Donc,
  \[
  N_{\ol{K}/\ol{F}}(a\ol{\tau}(a)^{-1})=a\ol{\tau}(a)^{-1}\cdot\ol{\tau}\big(a\ol{\tau}(a)^{-1}\big)=1\,.
  \]Cela implique que  $a\ol{\tau}(a)^{-1}=\alpha\ol{\tau}(\alpha)^{-1}$ pour un certain \'el\'ement $\alpha\in \ol{K}^*$. Alors, l'\'el\'ement $t:=a\alpha^{-1}\in\ell^*$ est fix\'e par $\ol{\tau}$, c'est-\`a-dire, $t\in \ell_0^*$. Ainsi on voit que  $a=\alpha t\in \ol{K}^*\ell_0^*$ comme desir\'e.
\end{proof}

Dans le contexte du \cref{3p6new}, Yanchevskii a \'egalement prouv\'e la trivialit\'e du groupe $\USK_1(D/F)$ lorsque $\ol{K}$ est un corps de nombres (\cite{Yanchevskii78}*{Cor.\;4.17}).
On g\'en\'eralise ce r\'esultat au th\'eor\`eme suivant.

\begin{thme}\label{2.3.7}
Soit $K$ un corps de valuation discr\`ete hens\'elien \`a corps r\'esiduel $\ol{K}$.
Supposons que $\car(\ol{K})=0$, $\mathrm{vcd}(\ol{K})\le 2$ et $\ol{K}$ n'a qu'un nombre fini d'ordres.

Alors on a $\USK_1(D/F)=1$ pour chaque $K$-alg\`ebre simple centrale \`a division $D$ qui a une $K/F$-involution unitaire.
\end{thme}
\begin{proof}[D\'emonstration]
On peut supposer que $\ol{K}$ est r\'eel et que $\deg(D)$ est une puissance de $2$,
car sinon le r\'esultat est imm\'ediat du \cref{3p6new}. Comme on a vu dans la \cref{remark on USK1 and SK1} (3)(iv), l'hypoth\`ese $\mathrm{vcd}(\ol{K})\le 2$ implique $\USK_1(\theta,\,\ol{D})=1$ pour toute involution unitaire $\theta$ sur $\ol{D}$. Vu le \cref{2.3.5}, on peut d'ailleurs supposer que $K/F$ est non ramifi\'ee, que $\ol{D}$ est non commutative et que $D/K$ est ramifi\'ee.

Prenons les notations du paragraphe \eqref{3p4new}. Par la vertu de la \cref{3p5new}, il suffit de prouver $\mathrm{PU}(\tau,\,D)=1$.

On met $\ell:=Z(\ol{D})$ et $\ell_0\ce \{a\in \ell\,|\,\ol{\tau}(a)=a\}$.
Si $\ell$ n'est pas r\'eel,
alors on a $\Nrd_{\ol{D}}(\ol{D}^*)=\ell^*$
et on peut raisonner comme dans la d\'emonstration du \cref{3p6new}.

Supposons $\ell$ r\'eel et notons $S_0$ l'ensemble de tous les ordres de $\ell$ dont la cl\^oture r\'eelle ne d\'eploie pas  $\ol{D}$.
Soit $S=\{ig\,|\,i\in S_0,\,g\in \mathrm{Gal}(\ell/\ol{K})\}$.
Comme dans la preuve du \cref{1.1.8}, on a
\[
\Nrd_{\ol{D}}(\ol{D}^*)=\ell_{S_0}^+=\ell^+_S=\{a\in \ell\,|\,i(a)>0 \text{ pour tout }\; i\in S\}\,.
\]
Soit $T_0$ l'ensemble des ordres de $\ell_0$ au dessous de $S$.

Nous savons d\'ej\`a que $\USK_1(\ol{\tau},\,\ol{D})=1$, d'o\`u $\Sigma_{\ol{\tau}}(\ol{D})=\Sigma'_{\ol{\tau}}(\ol{D})$.
Ainsi,
\[
\Nrd_{\ol{D}}\big(\Sigma_{\ol{\tau}}(\ol{D})\big)
=\Nrd_{\ol{D}}\big(\Sigma'_{\ol{\tau}}(\ol{D})\big)=(\ell_0)^+_{T_0}
=\{x\in\ell_0\,|\,i(x)>0 \text{ pour tout }\;i\in T_0\}\,.
\]

Pour prouver la trivialit\'e de $\mathrm{PU}(\tau,\,D)$,
on consid\`ere un \'el\'ement $a\in\ell^*$ tel que $\sigma(a)a^{-1}=b\ol{\tau}_2(b)^{-1}$ pour un certain $b\in\Nrd_{\ol{D}}(\ol{D}^*)$.
Il faut montrer $a\in \ol{K}^*\Nrd_{\ol{D}}\big(\Sigma_{\ol{\tau}}(\ol{D})\big)=\ol{K}^*(\ell_0)^+_{T_0}$.

Comme dans la preuve du \cref{3p6new}, on voit que $a\in \ol{K}^*\ell_0^*$. On peut donc supposer $a\in\ell_0^*$.

Maintenant, il suffit de trouver un \'el\'ement $c\in\ol{F}^*$ tel que $ca\in\Nrd_{\ol{D}}(\ol{D}^*)=\ell_S^+$.
En effet, cela implique que $ca\in\ell_S^+\cap \ell_0=(\ell_0)_{T_0}^+$,
 donc $a\in \ol{K}^*(\ell_0)^+_{T_0}=\ol{K}^*\Nrd_{\ol{D}}\big(\Sigma_{\ol{\tau}}(\ol{D})\big)$.

Pour voir l'existence de $c$,
soient $T$ l'ensemble des ordres de $\ol{F}$ au dessous de $T_0$ et
$T^{(i)}$ l'ensemble des ordres de $\ell_0$ au dessus de $i$ pour chaque $i\in T$.
On affirme que les \'el\'ements $j(a),\,j\in T^{(i)}$ ont tous le m\^eme signe.

Admettons cette affirmation pour l'instant. On peut alors d\'efinir le signe $\mathrm{sgn}_i(a)$ de $a$ en $i\in T$ comme le signe $\mathrm{sgn}(j(a))$ pour tout $j\in T^{(i)}$.
Puisque par l'hypoth\`ese $S$ est un ensemble fini,
l'ensemble $T$ est fini aussi.
Donc par l'approximation faible (\cite{Efrat2006book}*{Thm.\;10.1.7}),
il existe un \'el\'ement $c\in \ol{F}$ tel que $\mathrm{sgn}(i(c))=\mathrm{sgn}_i(a)$ pour tout $i\in T$.
Alors $ca\in\ell_S^+=\Nrd_{\ol{D}}(\ol{D}^*)$, comme d\'esir\'e.

Il reste \`a prouver notre affirmation concernant les signes des \'el\'ements $j(a)$, $j$ parcourant l'ensemble $T^{(i)}$ pour un $i\in T$ fix\'e.

Observons d'abord que les deux extensions (s\'eparables) $\ell_0/\ol{F}$ et $\ol{K}/\ol{F}$ sont disjointes, car $\ol{\tau}|_{\ol{K}}$ est non trivial. Donc tout $\ol{F}$-plongement de $\ell_0$ dans une cl\^oture alg\'ebrique de $\ol{F}$ s'\'etend en un $\ol{K}$-plongement de $\ell$. Soit  $\ol{F}_i$ la cl\^oture r\'eelle de $\ol{F}$ en $i$, et soit  $\ol{K}_i$ la cl\^oture r\'eelle de $\ol{K}$ en un ordre au dessus de $i$. Alors les \'el\'ements de $T^{(i)}$ correspondent aux $\ol{F}$-plongements de $\ell_0$ dans $\ol{F}_i$. On sait que tout  $\ol{F}$-plongement de $\ell_0$ dans $\ol{F}_i$ est la restriction d'un $\ol{K}$-plongement de $\ell$ dans $\ol{K}_i$. Donc, si $j\in T^{(i)}$ correspond au g\'en\'erateur $\sigma\in\mathrm{Gal}(\ell/\ol{K})$ (on rappelle que l'on a suppos\'e $\ell$ r\'eel), alors tout $j'\in T^{(i)}$ s'\'ecrit $j'=j\sigma^r$ pour un entier $r\ge 1$. Or, $\sigma(a)a^{-1}=b\ol{\tau}_2(b)^{-1}$ pour un $b\in\Nrd_{\ol{D}}(\ol{D}^*)$. En particulier, $\sigma(a)a^{-1}\in\Nrd_{\ol{D}}(\ol{D}^*)=\ell_S^+$. Cela donne $\sigma^2(a)\sigma(a)^{-1}=\sigma(\sigma (a)a^{-1})\in \ell_S^+$ et ainsi de suite. On en d\'eduit
\[
\sigma^r(a)a^{-1}=(\sigma^r(a)\sigma^{r-1}(a)^{-1})\cdot(\sigma^{r-1}(a)\sigma^{r-2}(a)^{-1})\cdots (\sigma(a)a^{-1})\in \ell_S^+\,,
\] d'o\`u
\[
j'(a)=j(\sigma^r(a))=j(\sigma^r(a)a^{-1})j(a)\;\in\;j(a)\ell_S^+\,.
\]Notre affirmation est ainsi d\'emontr\'ee. Ceci compl\`ete la d\'emonstration du th\'eor\`eme.
\end{proof}

\begin{rmke}
Soit $k$ un corps.
Alors $k$ ne poss\`ede qu'un nombre fini d'ordres s'il est un corps de normbres, ou si le sous-groupe des sommes non nulles de carr\'es dans $k^*$ est d'indice fini (\cite{Lam05}*{Chap.\;VIII, Exercise\;11}).
Par exemple, le corps $\mathbb{R}(\!(t_1)\!)\cdots(\!(t_n)\!)$ n'a qu'un nombre fini d'ordres pour chaque $n\in\mathbb{N}$
(voir aussi \cite{Lam05}*{Prop.\;VIII.4.11} pour une assertion plus pr\'ecise pour ce corps).

Par cons\'equent, dans le \cref{2.3.7} le corps r\'esiduel $\ol{K}$ peut \^etre $\mathbb{R}(\!(x)\!)$ ou $\mathbb{R}(\!(x)\!)(\!(y)\!)$.
\end{rmke}

\section{Propri\'et\'es d'approximation pour les groupes simplement connexes de type A}

Dans cette section, soit $C$ une courbe affine irr\'eductible normale  sur un corps $k$ et soit $K=k(C)$ le corps des fonctions de $C$.
Soit $\Omega$ l'ensemble des valuations discr\`etes de $K$ correspondant aux points ferm\'es de $C$.
Pour chaque $v\in\Omega$,
soit $K_v$ le compl\'et\'e de $K$ en $v$ et soit $\mathcal{O}_v$ l'anneau de valuation de $K_v$.
Si $D$ est une $K$-alg\`ebre,
on note $D_v\ce D\otimes_KK_v$ pour chaque $v\in \Omega$.

\begin{dfe}
Soit $X$ une vari\'et\'e alg\'ebrique sur $K$.
L'espace ad\'elique $X(\mathbf{A}_K)$ est l'espace produit restreint
des espaces topologiques $v$-adiques $X(K_v)$, $v\in\Omega$ par rapport aux
sous-espaces ouverts $X(\mathcal{O}_v)$ d\'efinis pour presque toute $v\in\Omega$.

On dit que $X$ satisfait \emph{l'approximation forte} (resp. \emph{l'approximation faible})
si $X(K)$ est dense dans l'espace ad\'elique $X(\mathbf{A}_K)$ (resp. dans l'espace produit $\prod_{v\in\Omega}X(K_v)$).
\end{dfe}

On peut utiliser la m\^eme m\'ethode de \cite{CTGP04Duke}*{Thm.\;4.7}
pour montrer que les groupes de type $\mathbf{A}$ satisfont l'approximation faible.

\begin{prop}\label{prop4.2}
Soit D une alg\`ebre simple centrale \`a division sur $K=k(C)$ et $n\ge 1$.
Supposons que $\SK_1(D_v)=1$ pour toute $v\in\Omega$.
Alors le groupe $G\ce \mathbf{SL}_n(D)$ satisfait l'approximation faible.
Si $n\ge 2$, alors $G$ satisfait l'approximation forte.
\end{prop}
\begin{proof}[D\'emonstration]
Soient $S\subset \Omega$ un sous-ensemble fini, $B=\mathrm{M}_n(D)$ et $(g_v)\in \prod_{v\in S}G(K_v)=\prod_{v\in S}\mathrm{SL}_1(B_v)$.
Par hypoth\`ese, pour toute $v\in S$ on a $\SK_1(B_v)=1$,
donc chaque $g_v$ est un produit fini de commutateurs dans $B_v^*$.
Mais l'ensemble $S$ est fini,
alors on peut supposer que
\[
g_v=(a_{v,1}b_{v,1}a_{v,1}^{-1}b_{v,1}^{-1})\cdots (a_{v,m}b_{v,m}a_{v,m}^{-1}b_{v,m}^{-1})\quad \text{avec } a_{v,i},\,b_{v,i}\in B_v^* \text{ en toute } v\in S\,,
\] pour un entier $m\ge 1$ ind\'ependent de $v$.
L'ensemble $B^*$, vu comme  l'ensemble des $K$-points d'un ouvert de Zariski dans un espace affine sur $K$, est dense dans $\prod_{v\in S}B_v^*$. En choissant des \'el\'ements $a_i,b_i\in B^*$ assez proches de $a_{v,i},b_{v,i}$ en toute $v\in S$, on obtient un point $g=(a_1b_1a_1^{-1}b_1^{-1})\cdots (a_mb_ma_m^{-1}b_m^{-1})\in G(K)$, approchant de $g_v$ en toute $v\in S$.
% l'approximation faible est invariant sous les applications birationelles.
Cela montre que $G(K)$ est dense dans $\prod_{v\in S}G(K_v)$.

Pour l'approximation forte, on prend une famille $(g_v)\in G(\mathbf{A}_K)$ et un sous-ensemble fini (assez grand) $S\subset\Omega$.
%Par construction, on peut supposer que $(g_v)\in \prod_{v\in S}G(K_v)\times \prod_{v\notin S}G(\mathcal{O}_v)$
%pour certain sous-ensemble fini $S\subset \Omega$.
On cherche un point $g\in G(K)$ qui est suffisamment proche de $g_v$ en toute $v\in S$ et tel que $g\in G(\mathcal{O}_v)$ pour toute $v\notin S$. Comme la multiplication est continue,
on peut supposer que pour une certaine place $w\in S$, $g_v\in G(K_v)=\mathrm{SL}_1(\mathrm{M}_n(D_v))$ est la matrice identiti\'e $I_n$ pour toute $v\in S$ sauf $v=w$.
Pour $n\ge 2$, le groupe $G(K_w)=\mathrm{SL}_1(\mathrm{M}_n(D_w))$ est engendr\'e par
des matrices \'el\'ementaires de la forme $I_n+\lambda_w E_{ij}$ (\cite{Draxl}*{\S\;20, Thm.\;3}), o\`u $\lambda_w \in D_w$, $1\le i\neq j\le n$, et $E_{ij}$ d\'esigne la matrice
dont tous les coefficients sont nuls sauf que celui \`a la position $(i,j)$ vaut 1.
Encore par la continuit\'e de la multiplication,
on peut donc supposer $g_w=I_n+\lambda_wE_{ij}$ pour un certain $\lambda_w\in D_w$. En identifiant $D$ avec l'ensemble des $K$-points d'un espace affine sur $K$, on est ramen\'e \`a voir qu'un espace affine sur $K$ v\'erifie l'approximation forte. Cette derni\`ere assertion revient au lemme chinois pour  des anneaux de Dedekind (\cite{SerreCL}*{\S\,I.3, Approximation Lemma}).
%De cette fa\c{c}on, nous r\'eduisons le probl\`eme d'approximation forte \`a celui des anneaux de Dedekind (\cite{Cohn2012}*{Thm.\;10.5.10}).
%Dans \cite{CTGP04Duke}*{Thm.\;4.7} il a \'et\'e expliqu\'e que l'approximation faible pour $G$ r\'esulte de la trivialit\'e des groupes de Whitehead r\'eduits locaux.
%Si $n\ge 2$,
%alors le groupe $G(K_v)=\mathrm{SL}_1(\mathrm{M}_n(D_v))$ est engendr\'e par
%des matrices \'el\'ementaires \`a coefficients dans $D_v\ce D\otimes_KK_v$ (\cite{Draxl}*{\S\;20, Thm.\;3}).
%On peut alors prouver l'approximation forte par la m\'ethode de Eichler comme expliqu\'ee dans la preuve de \cite{Scharlau85}*{Thm.\;10.5.1}.
\end{proof}

La combinaison de la \cref{prop4.2} avec les r\'esultats discut\'es dans la section 2
donne plusieurs exemples o\`u $G=\mathbf{SL}_n(D)$ satisfait l'approximation faible ou forte.
En particulier, on obtient :

\begin{thme}\label{thm4.3}
Soit D une alg\`ebre simple centrale \`a division sur $K=k(C)$.
Supposons $n\ge 2$ et $\car(k)=0$.
Alors le groupe $G=\mathbf{SL}_n(D)$ satisfait l'approximation forte dans les deux cas suivants :

\benuma
\item
$\mathrm{vcd}(k)\le 2$.

\item $k=k_0(\!(t)\!)$ pour un corps $k_0$ tel que $\mathrm{vcd}(k_0)=2$, et $D$ est non ramifi\'ee sur $C$ (c'est-\`a-dire, non ramifi\'ee en toute $v\in\Omega$).
\end{enumerate}
\end{thme}
\begin{proof}[D\'emonstration]
Par la \cref{prop4.2}, il suffit de v\'erifier que $\SK_1(D_v)=1$ pour toute $v\in\Omega$.
Le premier cas est \'evident par le \cref{1.1.8},
et le second cas suit en combinant le \cref{1.1.8} avec la \cref{remark2.4} sur le cas non ramifi\'e.
\end{proof}

\begin{ege}
Le corps $k$ dans le premier cas du \cref{thm4.3} peut \^etre une extension finie de l'un des corps suivants:
$\mathbb{Q}\,,\,\mathbb{R}(t),\,\mathbb{R}(\!(t)\!),\,\mathbb{R}(x,\,y),\,\mathbb{R}(\!(x)\!)(y),\,\mathbb{R}(\!(x,\,y)\!)\,,\,\mathbb{Q}_p$, etc. Il en est de m\^eme du corps $k_0$ dans le deuxi\`eme cas du  \cref{thm4.3}.
\end{ege}

\begin{rmke}
On montre dans un autre article avec J.\,Liu \cite{HuLiuTian23} que pour une certaine alg\`ebre des quaternions $Q$,
le groupe $\mathbf{SL}_1(Q)$ ne satisfait pas l'approximation forte.
\end{rmke}

\begin{prop}\label{prop4.6}
Soit D une alg\`ebre simple centrale \`a division sur une extension s\'eparable quadratique $L$ de $K=k(C)$.
Soient $\sigma$ une $L/K$-involution unitaire sur $D$ et $(V,\,h)$ une forme hermitienne non d\'eg\'en\'er\'ee sur $(D,\,\sigma)$.
Supposons que $\USK_1(D_v/K_v)=1$ pour toute $v\in\Omega$.

Alors le groupe $G\ce \mathbf{SU}(D,\,h)$ satisfait l'approximation faible.
Si de plus $h$ est isotrope, alors $G$ satisfait l'approximation forte.
\end{prop}
\begin{proof}[D\'emonstration]
%Pour l'approximation faible,
%on peut utiliser le m\^eme preuve que \cite{CTGP04Duke}*{Thm.\;4.7}.
%Quand $h$ est isotrope, on peut d\'emontrer l'approximation forte pour $G$ en suivant la m\'ethode de \cite{Kneser65-SA}*{\S\;5}
%(c'est-\`a-dire que l'approximation forte est une cons\'equence de l'approximation faible et de l'existence d'\'el\'ements unipotents); ou simplement en appliquant \cite{Gille09KT}*{Lemme\;5.6}.
Pour l'approximation faible,
consid\'erons la suite exacte
\[
1\to H\to R_{L|K}\GL_{1,D}\overset{\alpha}{\longrightarrow} R^1_{L|K}\mathbb{G}_m \to 1,
\]
o\`u la  fl\`eche  $\alpha$ est donn\'ee par $x\mapsto \Nrd(x)/\sigma(\Nrd(x))$.
Selon \cite{CM98}*{Prop.\;5.7}, le groupe $H$ est stablement birationnel \`a $G$ sur $K$.
Il suffit de montrer que $H$ satisfait l'approximation faible.
Par la construction de $H$, on a $H(K)=\{x\in D^*\,|\, \sigma(\Nrd(x))=\Nrd(x)\}$.
On note
\[
\Sigma\ce \langle x\in D^*\,|\, \sigma(x)=x\rangle\subset H(K)
\quad\text{ et }\quad
\Sigma_v\ce \langle x\in D_v^*\,|\, \sigma(x)=x\rangle.
\]
% et de la m\^eme mani\`ere pour $\Sigma_v\subset H(K_v)$.
Par ailleurs, l'hypoth\`ese $\USK_1(D_v/K_v)=1$ implique $H(K_v)=\Sigma_v$. Soit $S\subset\Omega$ un sous-ensemble fini, et soit $(h_v)\in \prod_{v\in S}H(K_v)$. On peut choisir un entier $m\ge 1$, ind\'ependent de $v\in S$, tel que  $h_v\in H(K_v)$ est un produit de pr\'ecis\'ement $m$ \'el\'ements $\sigma$-invariants de $D_v^*$ pour toute $v\in S$. Utilisant un argument similaire \`a celui dans la preuve de la \cref{prop4.2},  de la densit\'e de $D^*$ dans $\prod_{v\in S}D_v^*$ on peut d\'eduire que  $H(K)$ est dense dans $\prod_{v\in S}H(K_v)$.

Pour l'approximation  forte, l'hypoth\`ese que $h$ est isotrope signifie que le groupe $G$ est isotrope. Pour chaque $v\in \Omega$, on note
$G(K_v)^+$  le sous-groupe de $G(K_v)$ engendr\'e par les \'el\'ements unipotents. D'apr\`es \cite{Wall59}*{Thm.\;1}, on a $G(K_v)/G(K_v)^+\simeq \USK_1(D_v/K_v)$. Donc, par hypoth\`ese on a  $G(K_v)^+=G(K_v)$ pour chaque $v\in \Omega$. Par cons\'equent,
le groupe $G$ satisfait l'approximation forte par \cite{Gille09KT}*{Lemme\;5.6(2)}.
\end{proof}

\begin{thme}\label{thm4.7}
Soit $D$ une alg\`ebre simple centrale \`a division sur une extension s\'eparable quadratique $L$ de $K=k(C)$.
Soient $\sigma$ une $L/K$-involution unitaire sur $D$ et $(V,\,h)$ une forme hermitienne non d\'eg\'en\'er\'ee sur $(D,\,\sigma)$.
Supposons que $\car(k)=0$ et que $h$ est isotrope.

Alors le groupe $G=\mathbf{SU}(D,\,h)$ satisfait l'approximation forte dans chacun des cas suivants :

\benuma
\item\label{thm4.7.1}
$\mathrm{ind}(D)$ est sans facteurs carr\'es.

\item\label{thm4.7.2}
$\mathrm{vcd}(k)\le 1$.

\item\label{thm4.7.3}
$K$ n'est pas r\'eel et $\mathrm{vcd}(k)=2$.

\item\label{thm4.7.4}$K$ n'est pas r\'eel, $k=k_0(\!(t)\!)$ pour un corps $k_0$ tel que $\mathrm{vcd}(k_0)=2$,
et $D$ est non ramifi\'ee sur $C$ (c'est-\`a-dire, pour toute $v\in\Omega$,
$D_v\ce D\otimes_KK_v$ est non ramifi\'ee sur $K_v$).

\item\label{thm4.7.5}
$K$ est r\'eel et $k=\mathbb{R}(\!(x)\!)(\!(y)\!)$.
\end{enumerate}
\end{thme}
\begin{proof}[D\'emonstration]
D'apr\`es la \cref{prop4.6},
il suffit de v\'erifier que les groupes de Whitehead unitaires r\'eduits locaux sont triviaux.
Les cas \ref{thm4.7.1} et \ref{thm4.7.2} sont des cons\'equences de la \cref{remark on USK1 and SK1} (3).
Pour les cas \ref{thm4.7.3} et \ref{thm4.7.4}, on peut utiliser le \cref{2.3.5}.
Le cas \ref{thm4.7.5} est valable en vertu du \cref{2.3.7}.
\end{proof}

\begin{ege}
Le corps $k$ dans le cas \ref{thm4.7.3} du \cref{thm4.7} peut \^etre une extension finie de l'un des corps :  $\mathbb{Q}\,,\,\mathbb{R}(t),\,\mathbb{R}(\!(t)\!),\,\mathbb{R}(x,\,y),\,\mathbb{R}(\!(x)\!)(y),\,\mathbb{R}(\!(x,\,y)\!)\,,\,\mathbb{Q}_p$, etc. Il en est de m\^eme du corps $k_0$ dans le cas \ref{thm4.7.4} du  \cref{thm4.7}.
\end{ege}

\begin{rmke}
Soit $G$ un groupe simplement connexe absolument presque simple sur $K$.
Rappelons que $G$ satisfait l'approximation forte s'il est isotrope et $K$-rationnel, par \cite{Gille09KT}*{Cor.\;5.11}.

Supposons que $K=k(C)$ n'est pas r\'eel.
Si $k$ est une extension finie de $\mathbb{R}(\!(t)\!)$,
on peut utiliser \cite{CTGP04Duke}*{Thm.\;4.3} pour voir que $G$ est isotrope et $K$-rationnel s'il n'est pas de type A.
Il en va de m\^eme si $k$ est une extension finie de $\mathbb{R}(t)$ et $G$ n'est pas de type A ou $E_8$.
(Ici, on utilise \cite{Benoist19PIHES}*{Thm.\;0.12} pour v\'erifier la condition  p\'eriode-indice dans \cite{CTGP04Duke}*{Thm.\;4.3}.)
Pour ces deux types de corps $k$, les Th\'eor\`emes \ref{thm4.3} et \ref{thm4.7} couvrent le cas o\`u $G$ est isotrope de type A.

Lorsque $k$ est un corps $p$-adique, le groupe $G$ est $K$-rationnel s'il est isotrope et de type B, C, ou D classique (cf. \cite{Gille09KT}*{Thm.\;6.1}). Pour le type D classique, notons que toute alg\`ebre \`a division centrale sur $K$ admettant une involution orthogonale est de degr\'e $\le 4$, d'apr\`es un th\'eor\`eme de Saltman (\cite{Saltman97}) (voir \cite{PS14invent} dans le cas $p=2$).
\end{rmke}

\begin{rmke}
  Dans les Propositions\;\ref{prop4.2} et \ref{prop4.6}, l'assertion sur l'approximation faible a \'et\'e remarqu\'ee ind\'ependamment par V.\,Suresh.
Notons aussi que pour l'approximation faible on peut agrandir l'ensemble $\Omega$ en y rajoutant toutes les valuations discr\`etes $v$ (non seulement celles provenant de la courbe $C$) telles que les groups $\SK_1$ ou $\USK_1$ en $v$ sont triviaux.

Par exemple, si $k=\mathbb{Q}_p$ et $p\nmid \deg(D)$, on peut mettre dans $\Omega$ les valuations de $K$ correspondant aux points de codimension 1 des mod\`eles r\'eguliers de $C$ sur $\mathbb{Z}_p$.
\end{rmke}

\noindent \emph{Remerciements.}
Nous remercions Ting-Yu Lee pour des discussions utiles, et les rapporteurs pour des commentaires qui nous aident \`a am\'eliorer l'article.
Le premier auteur a b\'en\'efici\'e d'une subvention de la Guangdong Basic and Applied Basic Research Foundation
(no.\,2021A1515010396).

\

Yong HU

\medskip

Department of Mathematics

Southern University of Science and Technology

%No. 1088, Xueyuan Blvd., Nanshan district

Shenzhen 518055, China

 %, Guangdong,China

Email: huy@sustech.edu.cn

\

Yisheng TIAN

\medskip

Institute for Advanced Study in Mathematics, Harbin Institute of Technology

Harbin 150001, China

Email: tysmath@mail.ustc.edu.cn


\begin{bibsection}[R\'ef\'erences]
\begin{biblist}
\bib{BFP98}{article}{
  author={Bayer-Fluckiger, Eva},
  author={Parimala, Raman},
  title={Classical groups and the Hasse principle},
  journal={Ann. of Math. (2)},
  volume={147},
  date={1998},
  number={3},
  pages={651--693},
  issn={0003-486X},
  review={\MR {1637659}},
  doi={10.2307/120961},
}

\bib{Benoist19PIHES}{article}{
  author={Benoist, Olivier},
  title={The period-index problem for real surfaces},
  journal={Publ. Math. Inst. Hautes \'{E}tudes Sci.},
  volume={130},
  date={2019},
  pages={63--110},
  issn={0073-8301},
  review={\MR {4028514}},
  doi={10.1007/s10240-019-00108-7},
}

\bib{CM98}{article}{
   author={Chernousov, Vladimir},
   author={Merkurjev, Alexander},
   title={$R$-equivalence and special unitary groups},
   journal={J. Algebra},
   volume={209},
   date={1998},
   number={1},
   pages={175--198},
   issn={0021-8693},
   review={\MR{1652122}},
   doi={10.1006/jabr.1998.7534},
}

\bib{Cohn2012}{book}{
  author={Cohn, Paul},
  title={Basic algebra: groups, rings and fields},
  year={2012},
  publisher={Springer Science \& Business Media},
}

\bib{Colliot96Ktheory}{article}{
  author={Colliot-Th\'{e}l\`ene, Jean-Louis},
  title={Quelques r\'{e}sultats de finitude pour le groupe $SK_1$ d'une alg\`ebre de biquaternions},
  language={French, with English summary},
  journal={$K$-Theory},
  volume={10},
  date={1996},
  number={1},
  pages={31--48},
  issn={0920-3036},
  review={\MR {1373817}},
  doi={10.1007/BF00534887},
}

\bib{CTGP04Duke}{article}{
  author={Colliot-Th\'{e}l\`ene, Jean-Louis},
  author={Gille, Philippe},
  author={Parimala, Raman},
  title={Arithmetic of linear algebraic groups over 2-dimensional geometric fields},
  journal={Duke Math. J.},
  volume={121},
  date={2004},
  number={2},
  pages={285--341},
  issn={0012-7094},
  review={\MR {2034644}},
  doi={10.1215/S0012-7094-04-12124-4},
}

\bib{Draxl}{book}{
  author={Draxl, Peter},
  title={Skew fields},
  series={London Mathematical Society Lecture Note Series},
  volume={81},
  publisher={Cambridge University Press, Cambridge},
  date={1983},
  pages={ix+182},
  isbn={0-521-27274-2},
  review={\MR {696937}},
  doi={10.1017/CBO9780511661907},
}

\bib{Efrat2006book}{book}{
  author={Efrat, Ido},
  title={Valuations, orderings, and Milnor $K$-theory},
  series={Mathematical Surveys and Monographs},
  volume={124},
  publisher={American Mathematical Society, Providence, RI},
  date={2006},
  pages={xiv+288},
  isbn={0-8218-4041-X},
  review={\MR {2215492}},
  doi={10.1090/surv/124},
}

\bib{Ershov82MathSb}{article}{
  author={Ershov, Yurii Leonidovich},
  title={Henselian valuations of division rings and the group $SK_{1}$},
  language={Russian},
  journal={Mat. Sb. (N.S.)},
  volume={117(159)},
  date={1982},
  number={1},
  pages={60--68},
  issn={0368-8666},
  review={\MR {642489}},
}

\bib{Gille00}{article}{
  author={Gille, Philippe},
  title={Invariants cohomologiques de Rost en caract\'{e}ristique positive},
  language={French, with English and French summaries},
  journal={$K$-Theory},
  volume={21},
  date={2000},
  number={1},
  pages={57--100},
  issn={0920-3036},
  review={\MR {1802626}},
  doi={10.1023/A:1007839108933},
}

\bib{Gille09KT}{article}{
  author={Gille, Philippe},
  title={Le probl\`eme de Kneser--Tits},
  language={French, with French summary},
  note={S\'{e}minaire Bourbaki. Vol. 2007/2008},
  journal={Ast\'{e}risque},
  number={326},
  date={2009},
  pages={Exp. No. 983, vii, 39--81 (2010)},
  issn={0303-1179},
  isbn={978-285629-269-3},
  review={\MR {2605318}},
}

\bib{GSz17}{book}{
  author={Gille, Philippe},
  author={Szamuely, Tam\'{a}s},
  title={Central simple algebras and Galois cohomology},
  series={Cambridge Studies in Advanced Mathematics},
  volume={165},
  note={Second edition of [MR2266528]},
  publisher={Cambridge University Press, Cambridge},
  date={2017},
  pages={xi+417},
  isbn={978-1-316-60988-0},
  isbn={978-1-107-15637-1},
  review={\MR {3727161}},
}

\bib{Gille19book}{book}{
  author={Gille, Philippe},
  title={Groupes alg\'{e}briques semi-simples en dimension cohomologique $\le $ 2},
  language={French},
  series={Lecture Notes in Mathematics},
  volume={2238},
  note={With a comprehensive introduction in English},
  publisher={Springer, Cham},
  date={2019},
  pages={xxii+167},
  isbn={978-3-030-17271-8},
  isbn={978-3-030-17272-5},
  review={\MR {3972198}},
  doi={10.1007/978-3-030-17272-5},
}

\bib{GSz06}{book}{
   author={Gille, Philippe},
   author={Szamuely, Tam\'{a}s},
   title={Central simple algebras and Galois cohomology},
   series={Cambridge Studies in Advanced Mathematics},
   volume={101},
   publisher={Cambridge University Press, Cambridge},
   date={2006},
   pages={xii+343},
   isbn={978-0-521-86103-8},
   isbn={0-521-86103-9},
   review={\MR{2266528}},
   doi={10.1017/CBO9780511607219},
}

\bib{GB07}{article}{
  author={Grenier-Boley, Nicolas},
  title={On the triviality of certain Whitehead groups},
  journal={Math. Proc. R. Ir. Acad.},
  volume={107},
  date={2007},
  number={2},
  pages={183--193},
  issn={1393-7197},
  review={\MR {2358018}},
  doi={10.3318/PRIA.2007.107.2.183},
}

\bib{Hu2017IMRN}{article}{
  author={Hu, Yong},
  title={A cohomological Hasse principle over two-dimensional local rings},
  journal={Int. Math. Res. Not. IMRN},
  date={2017},
  number={14},
  pages={4369--4397},
  issn={1073-7928},
  review={\MR {3674174}},
  doi={10.1093/imrn/rnw149},
}

\bib{HuLiuTian23}{article}{
  author={Hu, Yong},
  author={Liu, Jing},
  author={Tian, Yisheng},
  title={Strong approximation and {H}asse principle for integral quadratic forms over affine curves},
  date={2023},
  note={arXiv:2312.08849},
 % label={HT23},
}

\bib{JW90}{article}{
  author={Jacob, Bill},
  author={Wadsworth, Adrian},
  title={Division algebras over Henselian fields},
  journal={J. Algebra},
  volume={128},
  date={1990},
  number={1},
  pages={126--179},
  issn={0021-8693},
  review={\MR {1031915}},
  doi={10.1016/0021-8693(90)90047-R},
}

\bib{Kneser65-SA}{article}{
  author={Kneser, Martin},
  title={Starke Approximation in algebraischen Gruppen. I},
  language={German},
  journal={J. Reine Angew. Math.},
  volume={218},
  date={1965},
  pages={190--203},
  issn={0075-4102},
  review={\MR {0184945}},
  doi={10.1515/crll.1965.218.190},
}

\bib{KMRT98}{book}{
  author={Knus, Max-Albert},
  author={Merkurjev, Alexander},
  author={Rost, Markus},
  author={Tignol, Jean-Pierre},
  title={The book of involutions},
  series={American Mathematical Society Colloquium Publications},
  volume={44},
  note={With a preface in French by J. Tits},
  publisher={American Mathematical Society, Providence, RI},
  date={1998},
  pages={xxii+593},
  isbn={0-8218-0904-0},
  review={\MR {1632779}},
  doi={10.1090/coll/044},
}

\bib{Lam05}{book}{
  author={Lam, Tsit Yuen},
  title={Introduction to quadratic forms over fields},
  series={Graduate Studies in Mathematics},
  volume={67},
  publisher={American Mathematical Society, Providence, RI},
  date={2005},
  pages={xxii+550},
  isbn={0-8218-1095-2},
  review={\MR {2104929}},
}

\bib{Lipnitskii82}{article}{
  author={Lipnitski\u {\i }, Valery Antonovich},
  title={Norms in fields of real algebraic functions and the reduced Whitehead group},
  language={Russian, with English summary},
  journal={Dokl. Akad. Nauk BSSR},
  volume={26},
  date={1982},
  number={7},
  pages={585--588, 668},
  issn={0002-354X},
  review={\MR {672002}},
}

\bib{Merkurjev95ProcSympos58}{article}{
  author={Merkurjev, Alexander},
  title={Certain $K$-cohomology groups of Severi-Brauer varieties},
  conference={ title={$K$-theory and algebraic geometry: connections with quadratic forms and division algebras}, address={Santa Barbara, CA}, date={1992}, },
  book={ series={Proc. Sympos. Pure Math.}, volume={58}, publisher={Amer. Math. Soc., Providence, RI}, },
  date={1995},
  pages={319--331},
  review={\MR {1327307}},
  doi={10.1016/j.jpaa.2010.02.006},
}

\bib{Merkurjev99crelle}{article}{
  author={Merkurjev, Alexander},
  title={Invariants of algebraic groups},
  journal={J. Reine Angew. Math.},
  volume={508},
  date={1999},
  pages={127--156},
  issn={0075-4102},
  review={\MR {1676873}},
  doi={10.1515/crll.1999.022},
}

\bib{MerkurjevSuslin82English}{article}{
  author={Merkurjev, Alexander},
  author={Suslin, Andrei},
  title={$K$-cohomology of Severi-Brauer varieties and the norm residue homomorphism},
  language={English},
  journal={Izv. Akad. Nauk SSSR Ser. Mat.},
  volume={46},
  date={1982},
  number={5},
  pages={1011--1046, 1135--1136},
  issn={0373-2436},
  review={\MR {675529}},
}



\bib{PS14invent}{article}{
    AUTHOR = {Parimala, Raman},
    author={Suresh, Venapally},
     TITLE = {Period-index and u-invariant questions for function fields
              over complete discretely valued fields},
   JOURNAL = {Invent. Math.},
    VOLUME = {197},
      YEAR = {2014},
    NUMBER = {1},
     PAGES = {215--235},
      ISSN = {0020-9910},
  review={\MR {3219517}},
       DOI = {10.1007/s00222-013-0483-y},
       URL = {http://dx.doi.org/10.1007/s00222-013-0483-y},
}

\bib{PlatonovYanchevskii73}{article}{
  author={Platonov, Vladimir},
  author={Yanchevski\u {\i }, Vyacheslav},
  title={The structure of unitary groups and of the commutants of a simple algebra over global fields},
  language={Russian},
  journal={Dokl. Akad. Nauk SSSR},
  volume={208},
  date={1973},
  pages={541--544},
  issn={0002-3264},
  review={\MR {0320175}},
}


\bib{Saltman97}{article}{
    AUTHOR = {Saltman, David},
     TITLE = {Division algebras over {$p$}-adic curves},
   JOURNAL = {J. Ramanujan Math. Soc.},
    VOLUME = {12},
      YEAR = {1997},
    NUMBER = {1},
     PAGES = {25--47},
      ISSN = {0970-1249},
   review={\MR {1462850}},
}

\bib{Scharlau85}{book}{
  author={Scharlau, Winfried},
  title={Quadratic and Hermitian forms},
  series={Grundlehren der Mathematischen Wissenschaften [Fundamental Principles of Mathematical Sciences]},
  volume={270},
  publisher={Springer-Verlag, Berlin},
  date={1985},
  pages={x+421},
  isbn={3-540-13724-6},
  review={\MR {770063}},
  doi={10.1007/978-3-642-69971-9},
}

\bib{SerreCL}{book}{
   author={Serre, Jean-Pierre},
   title={Local fields},
   series={Graduate Texts in Mathematics},
   volume={67},
   note={Translated from the French by Marvin Jay Greenberg},
   publisher={Springer-Verlag, New York-Berlin},
   date={1979},
   pages={viii+241},
   isbn={0-387-90424-7},
   review={\MR{554237}},
}

\bib{SerreCG94}{book}{
  author={Serre, Jean-Pierre},
  title={Cohomologie galoisienne},
  language={French},
  series={Lecture Notes in Mathematics},
  volume={5},
  edition={5},
  publisher={Springer-Verlag, Berlin},
  date={1994},
  pages={x+181},
  isbn={3-540-58002-6},
  review={\MR {1324577}},
  doi={10.1007/BFb0108758},
}

\bib{Soman19}{article}{
  author={Soman, Abhay},
  title={On triviality of the reduced Whitehead group over Henselian fields},
  journal={Arch. Math. (Basel)},
  volume={113},
  date={2019},
  number={3},
  pages={237--245},
  issn={0003-889X},
  review={\MR {3988818}},
  doi={10.1007/s00013-019-01344-3},
}

\bib{Suslin06SK1}{article}{
  author={Suslin, Andrei},
  title={$SK_1$ of division algebras and Galois cohomology revisited},
  conference={ title={Proceedings of the St. Petersburg Mathematical Society. Vol. XII}, },
  book={ series={Amer. Math. Soc. Transl. Ser. 2}, volume={219}, publisher={Amer. Math. Soc., Providence, RI}, },
  date={2006},
  pages={125--147},
  review={\MR {2276854}},
  doi={10.1090/trans2/219/04},
}

\bib{TignolWadsworth15}{book}{
  author={Tignol, Jean-Pierre},
  author={Wadsworth, Adrian R.},
  title={Value functions on simple algebras, and associated graded rings},
  series={Springer Monographs in Mathematics},
  publisher={Springer, Cham},
  date={2015},
  pages={xvi+643},
  isbn={978-3-319-16359-8},
  isbn={978-3-319-16360-4},
  review={\MR {3328410}},
}

\bib{Wall59}{article}{
    AUTHOR = {Wall, Gordon E.},
     TITLE = {The structure of a unitary factor group},
   JOURNAL = {Inst. Hautes \'{E}tudes Sci. Publ. Math.},
  FJOURNAL = {Institut des Hautes \'{E}tudes Scientifiques. Publications
              Math\'{e}matiques},
    NUMBER = {1},
      YEAR = {1959},
     PAGES = {7--23},
      ISSN = {0073-8301},
   MRCLASS = {22.00 (20.00)},
  MRNUMBER = {104764},
MRREVIEWER = {R. Ree},
 review={\MR {0104764}},
     %  URL = {http://www.numdam.org/item?id=PMIHES_1959__1__7_0},
}

\bib{WangShianghaw50AJM}{article}{
  author={Wang, Shianghaw},
  title={On the commutator group of a simple algebra},
  journal={Amer. J. Math.},
  volume={72},
  date={1950},
  pages={323--334},
  issn={0002-9327},
  review={\MR {34380}},
  doi={10.2307/2372036},
}

\bib{Yanchevskii74MatSb}{article}{
  author={Yanchevski\u {\i }, Vyacheslav},
  title={Simple algebras with involutions, and unitary groups},
  language={Russian},
  journal={Mat. Sb. (N.S.)},
  volume={93 (135)},
  date={1974},
  pages={368--380, 487},
  review={\MR {0439819}},
}

\bib{Yanchevskii78}{article}{
  author={Yanchevski\u {\i }, Vyacheslav},
  title={Reduced unitary $K$-theory and division algebras over Henselian discretely valued fields},
  language={Russian},
  journal={Izv. Akad. Nauk SSSR Ser. Mat.},
  volume={42},
  date={1978},
  number={4},
  pages={879--918},
  issn={0373-2436},
  review={\MR {508832}},
}

\bib{Yanchevskii04Whitehead}{article}{
  author={Yanchevski\u {\i }, Vyacheslav},
  title={Whitehead groups and groups of $R$-equivalence classes of linear algebraic groups of non-commutative classical type over some virtual fields},
  conference={ title={Algebraic groups and arithmetic}, },
  book={ publisher={Tata Inst. Fund. Res., Mumbai}, },
  date={2004},
  pages={491--505},
  review={\MR {2094122}},
}

\end{biblist}
\end{bibsection}
\end{document}